\documentclass[10pt]{amsart}
 \font\sr = msbm10 scaled \magstep 1
\newcommand{\Q}{\mbox{\sr Q}}
\newcommand{\R}{\mbox{\sr R}}

\newtheorem{thm}{Theorem}[section]
\newtheorem{lemma}[thm]{Lemma}
\newtheorem{defn}[thm]{Definition}
\newtheorem{nota}[thm]{Notation}
\newtheorem{rmk}[thm]{Remark}

\newtheorem{coro}[thm]{Corollary}
\newtheorem{examp}[thm]{Example}
\newtheorem{prop}[thm]{Proposition}
\newtheorem{assumptions}[thm]{Assumption}

\newcommand{\bpr}{\begin{proof}}
\newcommand{\epr}{\end{proof}}
\newcommand{\bt}{\begin{thm}}
\newcommand{\et}{\end{thm}}
\newcommand{\bl}{\begin{lemma}}
\newcommand{\el}{\end{lemma}}
\newcommand{\br}{\begin{rmk}}
\newcommand{\er}{\end{rmk}}
\newcommand{\bc}{\begin{coro}}
\newcommand{\ec}{\end{coro}}
\newcommand{\bd}{\begin{defn}}
\newcommand{\ed}{\end{defn}}
\newcommand{\bn}{\begin{nota}}
\newcommand{\en}{\end{nota}}
\newcommand{\bex}{\begin{examp}}
\newcommand{\eex}{\end{examp}}
\newcommand{\bp}{\begin{prop}}
\newcommand{\ep}{\end{prop}}
\newcommand{\ba}{\begin{assumptions}}
\newcommand{\ea}{\end{assumptions}}
\newcommand{\beq}{\begin{equation}}
\newcommand{\eeq}{\end{equation}}

\newcommand{\ds}{\displaystyle}

\begin{document}
\renewcommand{\baselinestretch}{0.9}

\date{\today}
{

\title[Stabilization of the wave equation]{Stabilization of the wave equation with moving boundary}

\author{Ka\"{\i}s Ammari}
\address{UR Analysis and Control of PDEs, UR13ES64, Department of Mathematics, Faculty of Sciences of Monastir,
University of Monastir, 5019 Monastir, Tunisia}
\email{kais.ammari@fsm.rnu.tn}

\author{Ahmed Bchatnia}
\address{UR  Analyse Non-Lin\'eaire et G\'eom\'etrie, UR13ES32, Department of Mathematics, Faculty of Sciences of Tunis, University of Tunis El Manar, Tunisia}
\email{ahmed.bchatnia@fst.utm.tn}

\author{Karim El Mufti}
\address{UR Analysis and Control of PDEs, UR13ES64, ISCAE, University of Manouba, Tunisia}
\email{karim.elmufti@iscae.rnu.tn}

\begin{abstract}
We deal with the wave equation with assigned moving boundary ($0<x<a(t)$) upon which Dirichlet-Neuman boundary conditions are satisfied, here 
$a(t)$ is assumed to move slower than the light and periodically. We give a feedback which guarantees the exponential decay of the energy. The proof relies on a reduction theorem \cite{ABE,Yoc}. At the end we give a remark on the moving-pointwise stabilization problem. 
\end{abstract}
\date{}

\maketitle

\tableofcontents


\section{Introduction and main result}
In this note, we analyze the stabilization property of solutions for the wave equation with a moving boundary. More precisely, we consider the following system:
\begin{equation}\label{sys}
\left\{\begin{array}{lll}
u_{tt}-u_{xx}=0 \;\;\;\;
\mbox{for} \;\;\;\; 0<x<a(t), \, t > 0,\\
u(0,t)=0 \;\;\;\; \mbox{and} \;\;\;\; u_t(a(t),t)+f(t)u_x(a(t),t)=0, \, t > 0,\\
u(x,0)=\phi(x), u_t(x,0)=\psi(x), \;\; 0<x<a(0), \\
\end{array}\right.
\end{equation}
$(\phi,\psi)\in H^1_l ((0,a(0)))\times L^2((0,a(0))),$ where $$H^1_l ((0,a(0))) = \left\{v\in H^1 ((0,a(0))), \, v(0) = 0\right\}.$$
 Here $a$ is a strictly positive real function which is continuous, 1-periodic and $f\in L^\infty\big(\mathbb{R}_+^*\big)$.

\medskip

Denote by
$$
E_u(t) = \frac{1}{2}\int_{0}^{a(t)}\left[\left|u_t(x,t)\right|^2 + \left|u_x (x,t)\right|^2 \right] \, dx
$$
the energy of the field $u$. Our major concern will be to detect the feedback $f(t)$ necessary to obtain the exponential decay of $E_u(t).$

\medskip

We start with some notations and known results. Let $\mbox{Lip} (\R)$ be the space of Lipschitz continuous functions on $\R$. We shall denote the Lipschitz constant of a function $F$ by
$$
L(F):=\sup_{x,y\in \R,x\neq y} \left | \frac{F(x)-F(y)}{x-y}\right
|.
$$
On the existence of solutions to the system (\ref{sys}), we refer the reader to \cite{gon}. We have the following proposition:

\begin{prop}
If $a\in\mbox{Lip} (\R)$, $L(a)\in [0,1)$, $a>0$, $f\in L^\infty\big(\mathbb{R}_+^*\big)$ and \\$(\phi,\psi)\in H_l^1((0,a(0)))\times L^2((0,a(0))),$ denote by $Q:=(0,a(t))\times \R_{+}$ and \\$Q_{\tau}:=(0,a(t))\times (0,\tau), \tau \in \R_{+}$. There exists a unique weak solution\footnote{$u\in H^1(Q_{\tau})$ is called a weak solution of (\ref{sys}) if $u_{tt}-u_{xx}=0$ in ${\mathcal D}^\prime(Q)$ and the boundary conditions are satisfied.} $u$ of the system (\ref{sys}). Moreover there exists $h\in H_{\mbox{loc}}^1(\R)\cap L^\infty (\R)$ such that
\begin{equation}
u(x,t)=h(t+x)-h(t-x)\quad\mbox{a.e. in $Q$},
\end{equation}
and $u\in L^{\infty}(Q)\cap H^1(Q_{\tau})$.
\end{prop}
We denote by $D_{p}$ the set of continuous functions and strictly increasing of the form $x+g(x)$, where $g(x)$ is a 1-periodic continuous function.

\medskip

We recall the following results.
\begin{prop} $\bigr($\cite[Herman]{Her} and \cite[Yamaguchi]{Ya2}$\bigr)$
Let $a$ be a 1-periodic function. Then
\begin{equation}\label{Fexpr}
F:=(I+a)\circ (I-a)^{-1}
\end{equation}
belongs to $D_{p}$. Moreover, the rotation number $\rho(F)$ defined by
$$
\rho(F)= \ds \lim_{n \rightarrow \infty} \frac{F^n(x)-x}{n}
$$
exists, and the limit is equal for all $x \in \R$.
\end{prop}
After, we construct a transformation of the time-dependent domain $[0, a(t)] \times \R$ onto $[0, \rho(F) / 2] \times \mathbb{R}$ that preserves the D'Alembertian form of the wave equation. For this purpose, we use the following proposition:

\begin{prop}$\bigr($\cite[Herman, section II]{Her}$\bigr)$
Assume that $a(t)$ is a 1-periodic function, $a(t)>0$, $a \in \mbox{Lip}(\R)$ such that $L(a)\in [0,1)$. Assume also that
$|a'(t)|<1$ for all $t \in \R$ and $\rho (F)\in\R\setminus\Q$ such that there exists a function $H \in D_{p}$ and
\begin{equation}\label{redF}
H^{-1} \circ F \circ H(\xi) = \xi + \rho(F).
\end{equation}
\end{prop}
Our main result is stated now as follows:

\begin{thm}[Exponential stability]
Let
\begin{equation}
\label{feedback} \displaystyle{f(t)=\frac{(\mu-1)H'(a(t)+t)+(\mu+1)H'(-a(t)+t)}{(1-\mu)H'(a(t)+t)+(\mu+1)H'(-a(t)+t)}}
\end{equation}
where $\mu$ is a nonnegative constant and assume that there exist $\lambda_1>0$ and $\lambda_2>0$ such that
\begin{equation}\label{hyp}
\lambda_1 \leq H^\prime (t) \leq \lambda_2, \, t \in \mathbb{R}.
\end{equation}
Then, in the case where $\mu \neq 1$, there exists a positive constant $C$ such that
\begin{equation}\label{ener}
E_u(t)\leq C e^{-\omega t} E_u(0),
\end{equation}
for every solution $u$ of (\ref{sys}) with initial data $(\phi,\psi)\in H_l^1((0,a(0)))\times L^2((0,a(0)))$ and where
$\omega = \ln \left(\ds \left|\frac{1 + \mu}{1- \mu}\right| \right)$.

\medskip
In the case $\mu=1$ which corresponds to $f(t)=1$, we obtain
\begin{equation}\label{null}
E_u(t)=0, \mbox{ for all } t\geq T_0=:\big(I+a\big)^{-1}\circ H^{-1}\left(\frac{3\rho(F)}{2}\right).
\end{equation}
\end{thm}
We give an example where assumption (\ref{hyp}) is guaranteed (see \cite{ABE, Gon1} for more details).
\begin{examp}\label{ex}
Let $a$ be continuous and 1-periodic on $\R$, $a>0$, be such that
\begin{equation}\label{exa}
a(t):=\left\{\begin{array}{ll} \alpha
t+\frac{\alpha(1-\alpha)(1+\beta)}{2(\alpha-\beta)} & \mbox{if
$\frac{\alpha(1+\beta)}{2(\alpha-\beta)}\leq t\leq
\frac{\alpha(1+\beta)-2\beta}{2(\alpha-\beta)},$}\\
\beta t-\beta+\frac{\alpha(1-\beta^2)}{2(\alpha-\beta)} & \mbox{if
$\frac{\alpha(1+\beta)-2\beta}{2(\alpha-\beta)} \leq
t\leq\frac{\alpha(3+\beta)-2\beta}{2(\alpha-\beta)}$,}
\end{array}
\right.
\end{equation}
with $-1 < \beta < 0 < \alpha < 1$.
\medskip
The function $F$ is defined by:
\[
F(x):=(I+a)\circ (I-a)^{-1}(x)=\left\{\begin{array}{ll}l_1x+F_0
& \mbox{if $0\leq x\leq x_0,$}\\
l_2x+F_0+1-l_2 & \mbox{if $x_0<x<1,$}
\end{array}
\right.
\]
with $l_1:=\frac{1+\alpha}{1-\alpha}$, $l_2:=\frac{1+\beta}{1-\beta}$, $F_0:=\frac{l_2(l_1-1)}{l_1-l_2}$ and
$x_0:=\frac{1-l_2}{l_1-l_2}$.

\medskip
We extend $F$ through the formula: $F(x+1)=F(x)+1$ for any $x\in \mathbb{R}.$\\
Also the rotation number is given by the expression:
\begin{equation}
\label{formrho} \rho (F)=\frac{\ln l_1}{\ln
\left(\frac{l_1}{l_2}\right)},
\end{equation}
and the function $H$ given by (\ref{redF}) is done by
$$H(x)=h_0\ln \left(|x+h_1| \right) + h_2,$$ where
$h_0=\frac{1}{\ln\left(\frac{l_1}{l_2}\right)}$, $h_1=\frac{l_2}{l_1-l_2}$ and
$h_2=-\ln \left(h_1 \right)$. $H$ satisfies the inequalities:
\begin{equation}\label{l2l1}
\frac{1}{\ln( \frac{l_1}{l_2})}\frac{l_1-l_2}{l_1}\leq H'(x)\leq
\frac{1}{\ln( \frac{l_1}{l_2})}\frac{l_1-l_2}{l_2}.
\end{equation}
Here $\displaystyle{f(t)=\frac{2a(t)+2\mu t+ 2\mu h_1}{2\mu a(t) + 2t + 2h_1}}$ and $T_0= \big(I+a\big)^{-1}\circ H^{-1} \left(\frac{3\rho(F)}{2}\right).$
\end{examp}

\begin{rmk}
For the moving pointwise stabilization of the wave equation, see section \ref{sec3}. 
\end{rmk}

\section{Proof of the main result} \label{sec2}

Before starting the proof, we begin by defining a domain transformation $$\Phi : \mathbb{R}^2 \rightarrow \mathbb{R}^2,$$ using $H$ given by
(\ref{redF}), as follows:
\begin{equation}
\label{trans}
\left\{
\begin{array}{ll}
\xi = (H(x + t) - H(-x + t))/ 2,\\
\tau = (H(x + t) + H(-x + t))/ 2,
\end{array}
\right.
\end{equation}
for $(x,t) \in \mathbb{R}^2$.

\begin{prop} $\bigr($\cite[Yamaguchi]{YaHi}$\bigr)$
The transformation $\Phi$ is a bijection of $[0, a(t)] \times \mathbb{R}$ to $[0, \rho(F)  / 2] \times \R$ and $\Phi$ maps the
boundaries $x=0$ and $x=a(t)$ onto the boundaries $\xi=0$ and
$\xi=\rho(F) / 2$.
\end{prop}
\begin{prop}$\bigr($\cite[Yamaguchi]{YaHi}$\bigr)$
Let $u(x,t)$ satisfying $(\partial_{t}^{2} - \partial_{x}^{2}) u(x,t)=0$ and $V(\xi,\tau)$ defined by $u (\Phi^{-1}(\xi,\tau))$.
Then the following identity holds
$$
(\partial_{t}^{2} - \partial_{x}^{2}) u(x,t) = K(\xi,\tau) (\partial_{\tau}^{2} - \partial_{\xi}^{2})V(\xi,\tau)
$$
where $K(\xi,\tau)$ is defined by
$$
4 H' \circ H^{-1}(\xi + \tau) H' \circ H^{-1}(- \xi + \tau) \circ H^{-1}(\xi + \tau).
$$
\end{prop}
Now we consider the system:
\begin{equation}\label{bou}
\left\{\begin{array}{lll}
u_{tt}-u_{xx}=0 \;\;\;\;
\mbox{for} \;\;\;\; 0<x<a(t), \, t > 0,\\
u(0,t)=0 \;\;\;\; \mbox{and} \;\;\;\; u_t(a(t),t)+f(t) \, u_x(a(t),t)=0, \, t > 0,\\
u(x,0)=\phi_1(x), u_t(x,0)=\psi_1(x), \;\; 0<x<a(0) \\
\end{array}\right.
\end{equation}
where $\displaystyle{f(t)=\frac{(\mu-1)H'(a(t)+t)+(\mu+1)H'(-a(t)+t)}{(1-\mu)H'(a(t)+t)+(\mu+1)H'(-a(t)+t)}}$ and $\mu$ is a nonnegative constant.
\begin{prop}\label{equiv}
The transformation of the system (\ref{bou}) is
\begin{equation}\label{stat}
\left\{\begin{array}{lll}
v_{\tau\tau}-v_{\xi\xi}=0 \;\;\;\;
\mbox{for} \;\;\;\; 0<\xi<\rho(F)/2, \, \tau > 0,\\
v(0,\tau)=0 \;\;\;\; \mbox{and} \;\;\;\; v_{\tau}(\rho(F)/2,\tau)+\mu v_{\xi}(\rho(F)/2,\tau)=0, \, \tau > 0,\\
v(\xi,0)=\phi(\xi), v_{\tau}(\xi,0)=\psi(\xi), \;\; 0<\xi<\rho(F)/2).
\end{array}\right.
\end{equation}
\end{prop}

\begin{proof}
We have:
\begin{equation}\label{rel}
\left\{\begin{array}{ll}
u_{x}(a(t),t)= V_{\xi}(\rho(F) /2,\tau)
\xi_{x}(a(t),t)+V_{\tau}(\rho(F) /2,\tau)\tau_{x}(a(t),t), \\
u_{t}(a(t),t)= V_{\xi}(\rho(F) /2,\tau)
\xi_{t}(a(t),t)+V_{\tau}(\rho(F) /2,\tau)\tau_{t}(a(t),t).
\end{array}\right.
\end{equation}

Starting from  (\ref{bou}) and make use of:
$$
\xi_{x} = (\partial_{x}\xi) = (\partial_{t}\tau) = \tau_{t}= [H'(x + t) + H'(-x + t)] / 2,
$$
$$
\xi_{t} = (\partial_{t}\xi) = (\partial_{x}\tau) = \tau_{x} = [H'(x + t) - H'(-x + t)] / 2,
$$
we conclude:
$$u_t(a(t),t)+\frac{\xi_t-\mu \tau_t}{\mu\tau_x- \xi_x}u_x(a(t),t)=0.$$
So we have
$$
\xi_tu_x(a(t),t)-\xi_xu_t(a(t),t)+\mu(\tau_xu_t(a(t),t)-\tau_t u_x(a(t),t))=0.
$$
Finally we get:
$v_{\tau}(\rho(F)/2,\tau)+\mu v_{\xi}(\rho(F)/2,\tau)=0.$

\medskip

Note that $f(t)=1$ in the special case $\mu =1$.
\end{proof}

The next lemma shows that the energy of the solution (\ref{sys}) and the energy of the corresponding static system are equivalent.
\begin{lemma}\label{second} $\bigr($\cite[Ammari et al.]{ABE}$\bigr)$
Under the assumption (\ref{hyp}), there are two positive constants $C_1$ and $C_2$ such that
\begin{equation}
C_1 E_V\left(H(a(t)+t)-\frac{\rho(F)}{2}\right)\leq E_u(t)\leq C_2 E_V\left(H(a(t)+t)-\frac{\rho(F)}{2}\right),
\end{equation}
where $E_V(\tau)$ is the energy of the field $V$ defined by:
$$
E_V(\tau)=\displaystyle\int^{\rho(F) /2}_{0} \left(\left|V_{\xi} (\xi,\tau)\right|^2 + \left|V_{\tau}(\xi,\tau)\right|^2\right) \,
d\xi.
$$
Note that we can write $H(a(t)+t)=a(t)+t+g\big(a(t)+t\big)$, where $g(x)$ is a 1-periodic continuous function.
\end{lemma}
Finally, the stabilization of the system (\ref{sys}) is a direct combination of Proposition \ref{equiv}, Lemma \ref{second} and the following Lemma \ref{exv}.
\begin{lemma}\label{exv} $\bigr($\cite[Cox and Zuazua]{coxzuazua}$\big)$
For $\mu \neq 1$, there exists a positive constant $C$ such that
\begin{equation}\label{expv}
E_{V}(\tau)\leq C e^{-\, \ln \left( \left|\frac{1 + \mu}{1- \mu}\right| \right) \,  \tau} E_V(0), \, \forall \, \tau > 0,
\end{equation}
where $V$ is the solution of the following system:
\begin{equation}
\left\{\begin{array}{lll}
V_{\tau\tau}-V_{\xi\xi}=0 \;\;\;\;
\mbox{for} \;\;\;\; 0<\tau<\rho(F)/2, \, \tau > 0,\\
V(0,\tau)=0 \;\;\;\; \mbox{and} \;\;\;\; V_{\tau}(\rho(F)/2,\tau)+\mu V_{\xi}(\rho(F)/2,\tau)=0, \, \tau > 0,\\
V(\xi,0)=\phi(\xi), V_{\tau}(\xi,0)=\psi(\xi), \;\; 0<\xi<\rho(F)/2). \\
\end{array}\right.
\end{equation}
It is well known that for $\mu = 1$, $E_{V}(t)=0$ for all $t\geq \rho(F),$ see \cite{coxzuazua} for more details.
\end{lemma}
\section{Moving pointwise stabilization} \label{sec3}
We consider the following problem:
\begin{equation}\label{pt}
\left\{\begin{array}{lll}
u_{tt}-u_{xx} + [f_1(t)u_t+f_2(t)u_x] \delta_{a(t)}=0 \;\;\;\;
\mbox{for} \;\;\;\; 0<x<b(t), \, t > 0,\\
u(0,t)=0 \;\;\;\; \mbox{and} \;\;\;\; u(b(t),t)=0, \, t > 0,\\
u(x,0)=\phi_1(x), u_t(x,0)=\psi_1(x) \;\; 0<x<b(0). \\
\end{array}\right.
\end{equation}
The aim of this section is to determine the functions $f_1$, $f_2$ and $b$ to get after transformation the vibrations of a string with the static pointwise damping and conclude the asymptotic behavior of the energy.
\begin{prop}\label{pp}
The transformation of the system:
\begin{equation}
\label{uu}
\left\{\begin{array}{lll}
u_{tt}-u_{xx}+ K \left(\frac{H(a(t)+t)-H(-a(t)+t)}{2},\frac{H(a(t)+t)+H(-a(t)+t)}{2} \right)\\\cdot[(\frac{1}{H'(a(t)+t)}-\frac{1}{H'(-a(t)+t)}) u_t+(\frac{1}{H'(a(t)+t)}+\frac{1}{H'(-a(t)+t)})u_x] \delta_{a(t)}=0 \\\;\;\;\;
\mbox{for} \;\;\;\; 0<x<b(t)=\Lambda_t^{-1}(1)-t, \, t > 0,\\
u(0,t)=0 \;\;\;\; \mbox{and} \;\;\;\; u(b(t),t)=0, \, t > 0,\\
u(x,0)=\phi_1(x), u_t(x,0)=\psi_1(x), \;\; 0<x<b(0) \\
\end{array}\right.
\end{equation}
is
\begin{equation}\label{statp}
\left\{\begin{array}{lll}
v_{\tau\tau}-v_{\xi\xi} + v_\tau \delta_{\frac{\rho(F)}{2}}=0 \;\;\;\;
\mbox{for} \;\;\;\; 0<\xi<1, \, \tau > 0,\\
v(0,\tau)=0 \;\;\;\; \mbox{and} \;\;\;\; v(1,\tau)=0, \, \tau > 0,\\
v(\xi,0)=\phi(\xi), v_{\tau}(\xi,0)=\psi(\xi), \;\; 0<\xi<1,
\end{array}\right.
\end{equation}
where $\Lambda_t$ is defined by $\Lambda_t(y)=\frac{H(y+t)-H(-y+t)}{2}.$
\end{prop}
\begin{proof}
We recall that if $u(x,t)$ satisfying $(\partial_{t}^{2} - \partial_{x}^{2}) u(x,t)=0$ and $v(\xi,\tau)$ defined by $u (\Phi^{-1}(\xi,\tau))$.
Then the following identity holds
\begin{equation}\label{k}
(\partial_{t}^{2} - \partial_{x}^{2}) u(x,t) = K(\xi,\tau) (\partial_{\tau}^{2} - \partial_{\xi}^{2})v(\xi,\tau)
\end{equation}
where $K(\xi,\tau)$ is defined by
$$
4 H' \circ H^{-1}(\xi + \tau) H' \circ H^{-1}(- \xi + \tau) \circ H^{-1}(\xi + \tau).
$$
On the other hand, make use of (\ref{rel}) we obtain:
\begin{equation}\label{v}
\begin{array}{lll}
v_\tau & = &\frac{\xi_t u_x-\xi_x u_t}{\xi^2_t-\xi^2_x}\\ &=&(\frac{1}{H'(a(t)+t)}-\frac{1}{H'(-a(t)+t)}) u_t+(\frac{1}{H'(a(t)+t)}+\frac{1}{H'(-a(t)+t)})u_x.
\end{array}
\end{equation}
We combine (\ref{k}) and (\ref{v}) to obtain the result of the Proposition \ref{pp}.
\end{proof}
\begin{rmk}
If we return to the Example \ref{ex} and after some computation we get that the transformation of the system:
\begin{equation}
\left\{\begin{array}{lll}
u_{tt}-u_{xx}+ \left(\frac{8}{a(t)+t+h_1}\cdot\frac{a(t)+t}{-a(t)+t+h_1} \right)\cdot [a(t)u_t+(t+h_1)u_x] \, \delta_{a(t)}=0 \\\;\;\;\;
\mbox{for} \;\;\;\; 0<x<b(t)=(t+h_1)e^{\frac{h_2}{h_0}}\tanh(\frac{1}{h_0}), \, t > 0,\\
u(0,t)=0 \;\;\;\; \mbox{and} \;\;\;\; u(b(t),t)=0, \, t > 0,\\
u(x,0)=\phi_1(x), u_t(x,0)=\psi_1(x), \;\; 0<x<b(0), 
\end{array}\right.
\end{equation}
is the system (\ref{statp}).
\end{rmk}

As above and according to \cite{AHT,AHTN} we have the following:

\begin{itemize}
\item
$
\ds \lim_{t \rightarrow + \infty} E_u (t) = 0, \, \forall \, (\phi_1,\psi_1) \in H^1_0(0,b(0)) \times L^2(0,b(0)) \Leftrightarrow \rho(F) \notin \mathbb{Q}.
$
\item
For any $\rho(F)/2 \in (0,1)$ the system (\ref{uu}) is not exponentially stable in $H^1_0(0,b(0)) \times L^2(0,b(0))$.
\item
For all $\rho(F)/2 \in \mathcal{S}$ \footnote{Denote by $\mathcal{S}$ the set of all numbers $\xi \in (0,1)$ such that $\xi \notin \mathbb{Q}$ and if $[0,a_1,...,a_n,...]$ is the expansion of $\xi$ as a continued fraction, then $(a_n)$ is bounded. Let us notice that $\mathcal{S}$ is is obviously uncountable and, by classical results on diophantine approximation (cf. \cite{CA}, p. 120), its Lebesgue measure is equal to zero. In particular, by Euler–Lagrange theorem (see Lang \cite{L}, p. 57) S contains all $\xi \in (0,1)$ such that $\xi$ is an irrational quadratic number (i.e. satisfying a second degree equation with rational coefficients).}
and for all $(\phi_1,\psi_1) \in \mathcal{D}$ we have according to \cite{AT,AN} that there exists 
$C > 0$ such that:
$$E_u (t) \leq \frac{C}{t}, \,  \forall \, t >0.$$
\item
If $\varepsilon > 0$ then, for almost all $\rho(F)/2 \in (0,1)$ and for all $(\phi_1,\psi_1) \in \mathcal{D}$ we have according to \cite{AT,AN} that there exists $C > 0$ such that:
$$E_u (t) \leq \frac{C}{t^{\frac{1}{1 + \varepsilon}}}, \,  \forall \, t >0,$$
where
$$
\mathcal{D} := \left\{(\varphi,\psi) \in [H^2(0,b(0)) \cap H^1_0(0,b(0))] \times H^1_0(0,b(0)), \, \right.
$$
$$
\left.
\left(\frac{1}{H^\prime(a(0))}-\frac{1}{H^\prime(-a(0))} \right) \, \psi (a(0)) + \left(\frac{1}{H^\prime(a(0))} + \frac{1}{H^\prime(-a(0))}\, \right) \, \frac{d\varphi}{dx} (a(0)) = 0 \right\}.
$$
\end{itemize}

\end{document}